\documentclass[14pt]{amsart}
\usepackage{latexsym, amssymb}
\usepackage{amsmath,amssymb,amsfonts}
\usepackage[usenames]{color}
\usepackage[all]{xy}
\usepackage{graphicx,color,tikz}
\usepackage{hyperref}

\newcommand{\be}{\begin{equation}}
\newcommand{\ee}{\end{equation}}
\newcommand{\beq}{\begin{eqnarray}}
\newcommand{\eeq}{\end{eqnarray}}

\newtheorem{thm}{Theorem} [section]
\newtheorem{cor}[thm]{Corollary}
\newtheorem{lem}[thm]{Lemma}

\theoremstyle{definition}

\theoremstyle{remark}

\numberwithin{equation}{section}
\begin{document}
\title[Unit roots of the unit root $L$-functions of Kloosterman family]
{Unit roots of the unit root $L$-functions of Kloosterman family}
\begin{abstract}
As a consequence of Wan's theorem  about Dwork's conjecture, the unit root $L$-functions of the $n$-dimensional Kloosterman family are $p$-adic meromorphic.
   By studying the symmetric power $L$-functions associated to the Kloosterman family, we prove that for each $0\le j\le n$, the unit root $L$-function coming from slope $j$ has a unique unit root which equals 1.

\end{abstract}

\author[L.P. Yang]{Liping Yang}
\address{College of Mathematics and Physics, Chengdu University of Technology, Chengdu 610059, P.R. China}
\email{yanglp2013@126.com}
\author[H.Zhang]{Hao Zhang$^*$}
\address{Shing-Tung Yau Center of Southeast University, Yifu Architecture Building, No.2, Sipailou, Nanjing 210096, P.R. China}
\email{zhanhgao@126.com}

\thanks{$^*$ H. Zhang is the corresponding author and supported by National Natural Science Foundation of China (Grant No. 12171261).
L.P. Yang is supported by National Natural Science Foundation of China (Grant No. 12201078).
}
%\date{\today}
\maketitle

\section{Introduction}

Let $p$ be a prime number and  $\mathbb{F}_q$ be the finite field of $q=p^a$ elements with characteristic $p$.
Fix a nonzero element $\bar{\lambda}$ in the algebraic closure of $\mathbb{F}_q$, denoted by $\bar{\mathbb{F}}_q$. Write $\deg(\bar{\lambda})=[\mathbb{F}_q(\bar{\lambda}):\mathbb{F}_q]$.
Let  $\zeta_p$ be a fixed primitive $p$-th root of unity.  For positive integers $m$ and $n$, the $n$-dimensional Kloosterman sum is defined by 
$${\rm Kl}_{m}(\bar{\lambda},n):=\sum_{\mathbf x\in (\mathbb F_{q_{\bar\lambda}^m}^\ast)^n} \psi\circ \text{Tr}_{\mathbb F_{q_{\bar\lambda}^m}/\mathbb F_q}(x_1+\cdots+x_n+\frac{\bar{\lambda}}{x_1\cdots x_n}),$$
where $ q_{\bar\lambda}:=q^{\text{deg}(\bar\lambda)}, $ and $\psi$ is a non-trivial additive character on $\mathbb F_q$ defined by
$$\psi(\bar a)=\zeta_p^{\text{Tr}_{\mathbb F_q/\mathbb F_p}(\bar a)}$$
for $\bar a\in \mathbb F_q$. 
Then we define the associated $L$-function by
$$ L(\bar{\lambda},n,T):=\exp \Big(\sum_{m=1}^{\infty}\frac{ {\rm Kl}_{m}(\bar{\lambda},n)T^m}{m} \Big). $$
It is well known that the $L$-function $L(\bar{\lambda},n,T)^{(-1)^{n+1}}$ is a polynomial of degree $n+1$ (\cite{AS1}).

By Sperber \cite{SP1}, there are $n+1$ algebraic integers
$\pi_0(\bar{\lambda}),\ldots, \pi_n(\bar{\lambda})$ satisfying $${\rm ord}_{q^{\deg(\bar{\lambda})}}\pi_j(\bar{\lambda})=j \text{  for } j=0,\cdots,n$$
such that
$$L(\bar{\lambda},n,T)^{(-1)^{n+1}}=(1-\pi_0(\bar{\lambda})T)\cdots (1-\pi_n(\bar{\lambda})T). $$ 
We call $j$ the the {\it slope} of $\pi_j(\bar{\lambda})$.

In section 2, we show that $q^{-j\deg(\bar{\lambda})}\pi_j(\bar{\lambda})$ is $p$-adic 1-unit (see Lemma \ref{lemma8} below), that is, $|q^{-j\deg(\bar{\lambda})}\pi_j(\bar{\lambda})- 1|_p<1$.
  Hence for $\kappa\in \mathbb{Z}_p$, we can define the unit root $L$-function coming from slope $j$ by
  $$L_{unit}(j,\kappa, T):=\prod_{\bar{\lambda}\in |\mathbb {G}_m/\mathbb{F}_q|}\frac{1}{1-(q^{-j\deg(\bar{\lambda})}\pi_j(\bar{\lambda}))^{\kappa}T^{\deg(\bar{\lambda})}},$$
where $\bar{\lambda}$ runs over all closed points of $\mathbb {G}_m/\mathbb{F}_q$.
 Wan's theorem \cite{Wan99} about Dwork's conjecture shows that the unit root $L$-function $L_{unit}(j,\kappa, T)$ is $p$-adic meromorphic. A natural question is to study the zeros and poles of the unit root $L$-function.

Dwork \cite{Dw74} studied the 1-dimensional Kloosterman sums using $p$-adic Bessel function. In this case, Robba \cite{Ro86} first studied its $k$-th symmetric power $L$-functions using $p$-adic method and proved it is a polynomial with a unique unit root $T=1$. Using Robba's result, Haessig \cite{H17} showed that $T=1$ is also the unique unit root of the unit root $L$-function coming from slope 0.

On the other hand, Adolphson and Sperber \cite{AS2} proved the $L$-function associated to arbitrary exponential sum on the torus has a unique unit root, which is $1$-unit. Then the unit root $L$-function coming from slope $0$ can be defined for any family. For the family  with a lower deformation condition, Haessig and Sperber \cite{HS17} proved that it also has a unique unit root and expressed the unit root in terms of certain hypergeometric functions. Their results can not apply to the Kloosterman family since it does not satisfy the lower deformation condition.

To extend Haessig's result \cite{H17}  to the $n$-dimensional Kloosterman sums, we prove that the unit root $L$-function coming from arbitrary slope has $1$ as its unique unit root. We have the following result.

\begin{thm}\label{thm1}
Let $p\ge n+3$ and $\kappa\in \mathbb{Z}_p$. For each integer $0\le j\le n$,  the unit root $L$-function $L_{unit}(j,\kappa, T)$ coming from slope $j$ has a unique unit root $T=1$.
\end{thm}

This paper is organized as follows. In section 2, we introduce the Dwork's cohomology to study the Kloosterman sums. In section 3, we investigate $L_{unit}(j,\kappa, T)$ by studying certain infinity symmetric power $L$-functions and show that $L_{unit}(j,\kappa, T)$ has a unique unit root. In section 4, we prove the unique unit root is $1$ by using dual theory.

\section{Dwork's cohomology}

For $p\geq 3$, fix an algebraic closure $\overline{\mathbb Q}_p$ of $\mathbb Q_p$, all finite extensions of $\mathbb Q_p$ are taken in $\overline{\mathbb Q}_p$ throughout the paper. For $q=p^a$, let $\mathbb{Q}_q$ be the unramified extension of $\mathbb{Q}_p$ of degree $a$ and $\mathbb{Z}_q$ be its ring of integers. Denote by $\text{ord}_p$ the valuation on $\overline{\mathbb Q}_p$ normalized by $\text{ord}_pp=1$.
 In the following, we use the multi-index notation. For $\mathbf{v}=(v_1,\cdots,v_n)\in \mathbb Z^n$, $\mathbf{x}=(x_1,\cdots,x_n)$, write  $\mathbf{x}^{\mathbf{v}}=x_1^{v_1}\cdots x_n^{v_n}$.
Let $\pi \in \overline{\mathbb Q}_p$ be an element satisfying $\pi^{p-1}=-p$, $\Omega=\mathbb Q_q(\pi)$, and let $R$ be the ring of integers of $\Omega$.

For $m\in \mathbb{Z}_{\ge 0}$, we define
 $$\mathcal{O}_{0,p^m}:=\{\sum_{r=0}^{\infty} C(r)\pi^{\frac{(n+1)r}{p^m}}\Lambda^r : C(r)\in R, C(r)\rightarrow 0\ {\rm as}\  r\rightarrow \infty \}.$$
Then $\mathcal{O}_{0,p^m}$ is a Banach space with respect to the norm defined by
 $$|\sum_{r=0}^{\infty}C(r)\pi^{\frac{(n+1)r}{p^m}}\Lambda^r|=\sup_{r\ge 0}|C(r)|. $$

 Let $s(\mathbf{v})=\max \{0,-v_1,\cdots,-v_n \}$ and  $w(\mathbf{v})=\sum_{i=1}^n v_i+(n+1)s(v)$.
 Define an $\mathcal{O}_{0,p^m}$-algebra
 $$\mathcal{C}_{0,p^m}:=\{\sum_{\mathbf{v}\in \mathbb{Z}^n}\zeta({\mathbf v}) \pi^{w({\mathbf v})}\Lambda^{p^m s(\mathbf{v}) }\mathbf{x}^{\mathbf{v}}: \zeta(\mathbf{v})\in \mathcal{O}_{0,p^m},|\zeta(\mathbf{v})|\rightarrow 0\ {\rm as}\ w(\mathbf{v})\rightarrow \infty \}. $$
It is also a Banach algebra with respect to norm
 $$|\sum_{\mathbf{v}\in \mathbb{Z}^n}\zeta(\mathbf{v}) \pi^{w(\mathbf{v})}\Lambda^{p^m s({\mathbf v}) }\mathbf{x}^{\mathbf{v}}|:=\sup_{{\mathbf v}\in \mathbb{Z}^n}\{ |\zeta(\mathbf{v})|\} .$$

When $m=0$, we simply write $\mathcal{O}_{0,1}$ and $\mathcal{C}_{0,1}$ by $\mathcal{O}_0$ and $\mathcal{C}_0$ respectively. Let $f(\Lambda,\mathbf{x})=x_1+\cdots+x_n+\frac{\Lambda}{x_1\cdots x_n}$. For $i=1,\cdots,n$, we define the following operators acting on $\mathcal{C}_{0,p^m}$:
\begin{eqnarray*}D_{i,\Lambda^{p^m}}:&=&x_i\frac{\partial}{\partial x_i}+\pi (x_i-\frac{\Lambda^{p^m}}{x_1\cdots x_n})\\
&=&\exp\big(-\pi f(\Lambda^{p^m},\mathbf{x})\big)\circ x_i \frac{\partial}{\partial x_i}\circ \exp\big(\pi f(\Lambda^{p^m},\mathbf{x})\big);\\
\nabla_{\Lambda^{p^m}}:&=&\Lambda \frac{\partial}{\partial \Lambda}+\frac{p^m \pi  \Lambda^{p^m}}{x_1\cdots x_n}\\
 &=&\exp\big(-\pi f(\Lambda^{p^m},\mathbf{x})\big)\circ \Lambda \frac{\partial}{\partial \Lambda}\circ \big(\exp\pi f(\Lambda^{p^m},\mathbf{x})\big).
\end{eqnarray*}

Note that $\frac{\partial}{\partial x_i}$ commutes with $\frac{\partial}{\partial \Lambda}$, then $D_{i,\Lambda^{p^m}}$ commutes with $\nabla_{\Lambda^{p^m}}$. One checks that 
\begin{equation}\label{basis}
\nabla_{\Lambda^{p^m}}^j(1)\equiv \pi^j x_1\cdots x_j \mod \sum_{i=1}^nD_{i,\Lambda^{p^m}}\mathcal{C}_{0,p^m}.
\end{equation}
Using the same argument as  \cite[Theorem 3.2]{HS172}, we have that the quotient
$$\mathcal{H}_{\Lambda^{p^m}}=\mathcal{C}_{0,p^m}/\sum_{i=1}^nD_{i,\Lambda^{p^m}}\mathcal{C}_{0,p^m}$$
is a free $\mathcal{O}_{0,p^m}$-module of rank $n+1$ with basis $\{1, \bar{\nabla}_{\Lambda^{p^m}}(1),\cdots,\bar{\nabla}_{\Lambda^{p^m}}^n(1)\}$.
In the following, we fix the basis $\{1, \bar{\nabla}_{\Lambda^{p^m}}(1),\cdots,\bar{\nabla}_{\Lambda^{p^m}}^n(1)\}$
 and let $e_0=1$ and $e_i=\bar{\nabla}_{\Lambda^{p^m}}^i(1)$ for $i=1,\cdots,n$. From (\ref{basis}) we see that the basis $\{e_i\}$ does not rely on $m$.

\subsection{The Frobenius map}

For $m\in \mathbb{Z}_{\ge 0}$ and $b,c\in \mathbb{R}$, $b>0$, we define the space
 $$K_{p^m}(b;c):=\{\sum_{r\ge p^m\cdot s(\mathbf{v})} a(r,\mathbf{v})\Lambda^r\mathbf{x}^{\mathbf{v}}: a(r,\mathbf{v})\in \Omega, {\rm ord}_pa(r,\mathbf{v})\ge bw_{p^m}(r,\mathbf{v})+c \}, $$
where $w_{p^m}(r,\mathbf{v}):=\sum_{i=1}^nv_i+\frac{(n+1)r}{p^m} $. One checks that
$$\mathcal{C}_{0,p^m}\subset K_{p^m}(\frac{1}{p-1};0),\ \  K_{p^m}(b,0)\subset C_{0,p^m} \text{ for } b>\frac{1}{p-1}.$$
Let $$K_{p^m}(b):=\cup_{c\in \mathbb{R}} K_{p^m}(b;c).$$
When $m=0$ we use $K(b;c)$ and $K(b)$ to denote $K_{1}(b;c)$ and $K_{1}(b)$, respectively .

 Let $\Omega[[\Lambda]]$ be the power series ring over $\Omega$. The coefficient rings of $K_{p^m}(b;c)$ and $K_{p^m}(b)$ are defined by
$$L_{p^m}(b;c):=K_{p^m}(b;c)\cap \Omega[[\Lambda]]=\Big\{\sum_{r\ge 0}a(r)\Lambda^r: a(r)\in \Omega, {\rm ord}_pa(r)\ge \frac{b(n+1)r}{p^m}+c \Big\} $$
and
$$L_{p^m}(b):=\cup_{c\in \mathbb{R}}L_{p^m}(b;c) .$$
If $m=0$, then we simply write $L_{1}(b;c)$ and $L_{1}(b)$ by $L(b;c)$ and $L(b)$, respectively.

Define $\psi_{\mathbf x}$ acting on the power series ring involving $\Lambda$ and $\mathbf x$ by
$$\psi_{\mathbf x}\big(\sum a(r, \mathbf v)\Lambda^r\mathbf x^{\mathbf v}\big)=\sum a(r, p\mathbf v)\Lambda^r\mathbf x^{\mathbf v}.$$
One checks that $\psi_{\mathbf x}\big(K_{p^m}(b;0)\big)\subset K_{p^{m+1}}(pb;0)$ for $m\geq 0$.
Let $\theta(t):=\exp \pi(t-t^p)$ be a Dwork splitting function. Write $\theta(t)=\sum_{j\ge 0}\theta_j t^j$, then  $\text{ord}_p(\theta_j)\ge \frac{p-1}{p^2}j.$ It follows from the estimates that for $m\ge 1$,
\begin{eqnarray*}
F(\Lambda,\mathbf x)&:=&\theta(x_1)\cdots \theta(x_n)\theta(\frac{\Lambda}{x_1\cdots x_n})\in K((p-1)/p^2;0),\\
F_m(\Lambda,\mathbf x)&:=&\prod^{m-1}_{i=0}F(\Lambda^{p^i}, \mathbf x^{p^i})\in K((p-1)/p^{m+1};0).
\end{eqnarray*}
For $m\ge1$, define
$$\alpha_m=\psi_{\mathbf x}^m\circ F_m(\Lambda,\mathbf x)=\exp\big(-\pi f(\Lambda^{p^m},\mathbf x)\big)\circ\psi_{\mathbf x}^m\circ \exp\big(\pi f(\Lambda,\mathbf x)\big).$$
Note that $\frac{p-1}{p}>\frac{1}{p-1}>\frac{p-1}{p^2}$ for $p>2$. Then $\alpha_m$ is a map from
$\mathcal{C}_0$ to $\mathcal{C}_{0,p^m}$. It is indeed the composition
$$\mathcal{C}_0\subset K(\frac{1}{p-1};0)\hookrightarrow K(\frac{p-1}{p^{m+1}};0)\stackrel{F_{m}}{\longrightarrow} K(\frac{p-1}{p^{m+1}};0)
\stackrel{\psi_{\mathbf {x}}^m}{\longrightarrow} K_{p^m}(\frac{p-1}{p};0)\subset \mathcal{C}_{0,p^m}.$$

For $i=1,\cdots,n$, one verifies from the definition that
$$\alpha_m\circ D_{i,\Lambda}=p^mD_{i,\Lambda^{p^m}}\circ\alpha_m.$$
Hence $\alpha_m$ induces a map from $\mathcal{H}_\Lambda$ to $\mathcal{H}_{\Lambda^{p^m}}$, which is denoted by $\bar\alpha_m$.

Let $\mathcal A_1(\Lambda)=(a_{ij}\big(\Lambda)\big)_{0\le i,j\le n}$ be the matrix of $\bar\alpha_1:\mathcal{H}_\Lambda\to \mathcal{H}_{\Lambda^{p}}$ with respect to the fixed basis, that is
 $$\bar\alpha_1\big(e_0,\cdots,e_{n}\big)=\big(e_0,\cdots,e_{n}\big)\mathcal A_1(\Lambda).$$ 
From the definition, one checks that $\alpha_m$ is represented by 
\begin{equation}\label{eqn005}
\mathcal A_m(\Lambda):=\mathcal A_1(\Lambda^{p^{m-1}})\cdots\mathcal A_1(\Lambda^p)\mathcal A_1(\Lambda)
\end{equation}
 relative to the fixed basis.
It follows from \cite[Theorem 1.3.9]{SP0} and \cite[Proposition 2.5.9]{SP0} that
\begin{equation}\label{eqnaij}
a_{ij}(\Lambda)\in L((p-1)/p^2;i).
\end{equation}
 Then  we apply  \cite[Theorem 4.2.11]{SP0} and \cite[Proposition 2.1]{SP1} to get the following description of the entries of $\mathcal A_1(\Lambda)$.

 \begin{lem}\label{lemma3}
 If $p\ge n+3$, then the constant terms of Frobenius matrix $\mathcal{A}_1(\Lambda)$ satisfy that
\begin{align}\label{eqn1}
&\left\{
  \begin{array}{ll}
a_{ij}(0)=0, &  \hbox{for  $j>i$};\\
 a_{ii}(0)=p^i, & \hbox{for $0\le i\le n$}.
\end{array}
\right.
\end{align}
Furthermore, for  $\lambda\in \bar{\mathbb{Q}}_p$ such that ${\rm ord}_p\lambda\ge 0$, we have
\begin{align}\label{eqn2}
&\left\{
  \begin{array}{ll}
a_{ii}(\lambda)/p^{i} \equiv 1 \mod \pi, &  \hbox{for $i=0,1,\ldots,n$};\\
  a_{ij}(\lambda)/p^{i}\equiv 0 \mod \pi, &  \hbox{for $j>i$}.
  \end{array}
\right.
\end{align}
and
 \begin{align}\label{eqn3}
 {\rm ord}_p a_{ij}(\lambda)\ge i.
 \end{align}
 \end{lem}

\subsection{Kloosterman sums and $L$-functions}
Recall that for each positive integer $m$ and $\bar\lambda\in \bar{\mathbb F}_q^\ast$, the Kloosterman sums ${\rm Kl}_{m}(\bar{\lambda},n)$ is
$${\rm Kl}_{m}(\bar{\lambda},n)=\sum_{\mathbf x\in (\mathbb F_{q_{\bar\lambda}^m}^\ast)^n} \psi\circ \text{Tr}_{\mathbb F_{q_{\bar\lambda}^m}/\mathbb F_q}(f(\bar{\lambda},\mathbf x)).$$
Let $$L(f(\bar\lambda,\mathbf x), T):=\exp\big(\sum_{m=1}^\infty \frac{{\rm Kl}_{m}(\bar{\lambda},n) T^m}{m}\big).$$
be the $L$-function for the Kloosterman sums.

Let $\lambda$ be the Techm\"uller lifting of $\bar\lambda$ in $\overline{\mathbb Q}_p$.
We regard the field extension $\Omega(\lambda)$ as an $\mathcal{O}_0$-algebra via the homomorphism
$$\mathcal{O}_0\to \Omega(\lambda),\ \ \ \ \Lambda\to \lambda.$$
Note that $\alpha_{a\text{deg}(\bar\lambda),\lambda}:=\psi_{\mathbf{x}}^{a\text{deg}(\bar{\lambda})}\circ F_{a\text{deg}(\bar{\lambda})}(\lambda,\mathbf{x})$ induces a map denoted by $\bar\alpha_{a\text{deg}(\bar\lambda),\lambda}$ on $\mathcal{H}_{\Lambda}\otimes_{\mathcal{O}_0}\Omega(\lambda)$ since $\lambda^{q^{\text{deg}(\bar\lambda)}}=\lambda$. By Dwork's trace formula \cite[Theorem 2.2]{AS1} and \cite[Lemma 2.8]{Fu}, we have
\begin{equation}
{\rm Kl}_{m}(\bar{\lambda},n)=(-1)^n\text{Tr}\big(\bar\alpha^m_{a\text{deg}(\bar\lambda),\lambda}|_{\mathcal{H}_{\Lambda}\otimes_{\mathcal{O}_0}\Omega(\lambda)}\big).
\end{equation}
In other words, we have
\begin{equation}
L(f(\bar\lambda,\mathbf x), T)^{(-1)^{n+1}}=\text{det}\big(1-T\bar\alpha_{a\text{deg}(\bar\lambda), \lambda}|_{\mathcal{H}_{\Lambda}\otimes_{\mathcal{O}_0}\Omega(\lambda)}\big)=\det\big(1-T\mathcal A_{a\text{deg}(\bar\lambda)}(\lambda)\big).
\end{equation}
By \cite[Theorem 3.10]{AS1} and \cite[Theorem 2.35]{SP1}, $L(f(\bar\lambda,\mathbf x), T)^{(-1)^{n+1}}$ is a polynomial of degree $n+1$, and can be written by
$$L\big(f(\bar\lambda,\mathbf x), T\big)^{(-1)^{n+1}}=\big(1-\pi_0(\bar\lambda )T\big)\cdots\big(1-\pi_n(\bar\lambda) T\big)$$
with ${\rm ord}_{q^{\deg(\bar{\lambda})}}\pi_j(\bar{\lambda})=j$ for $j=0,\cdots, n$.

 \subsection{Exterior power}
Let $i$ be an integer such that $1\le i \le n+1$.
For $i\ge 2$, the exterior power $\wedge^i \mathcal{H}_{\Lambda^{p^m}}$ is the
free $\mathcal{O}_{0,p^m}$-module with basis
$$\{ e_{u_1}\wedge \cdots \wedge e_{u_i}:0\le u_1<\cdots<u_i\le n\}.$$
Define a map $\phi_{m,i}: \wedge^i \mathcal{H}_{\Lambda}\rightarrow \wedge^i \mathcal{H}_{\Lambda^{p^m}}$ by
\begin{equation}\label{eqnphi}\phi_{m,i}(e_{u_1}\wedge\cdots \wedge e_{u_i})=p^{-mi(i-1)/2} \bar{\alpha}_{m}(e_{u_1} )\wedge \cdots\wedge \bar{\alpha}_{m}(e_{u_i} ).
\end{equation}
From (\ref{eqnaij}), it is well-defined. For $m=a\deg(\bar{\lambda})$, we denote by $\phi_{a\text{deg}(\bar{\lambda}),i,\lambda}$ the specialization of $\phi_{m,i}$  at $\lambda$ acting on $\wedge^i\mathcal{H}_{\Lambda}\otimes_{\mathcal{O}_0} \Omega(\lambda)$.

\begin{lem}\label{lemma1}
View $\{ e_{u_1}\wedge \cdots \wedge e_{u_i}:0\le u_1<\cdots<u_i\le n\}$ as the basis of $\wedge^i \mathcal{H}_{\Lambda}\otimes_{\mathcal{O}_0} \Omega(\lambda)$,
we have that $\phi_{a\text{deg}{(\bar{\lambda})},i,\lambda}(e_{0}\wedge \cdots \wedge e_{i-1})\equiv e_{0}\wedge \cdots \wedge e_{i-1} \mod \pi$ and
$\phi_{a\text{deg}{(\bar{\lambda})},i,\lambda}(e_{v_1}\wedge \cdots \wedge e_{v_i})\equiv 0 \mod \pi$ otherwise.
\end{lem}

\begin{proof}

Fix the basis $\{e_0,\cdots,e_n\}$ of $\mathcal{H}_{\Lambda}\otimes_{\mathcal{O}_0} \Omega(\lambda)$. The endomorphism $\bar{\alpha}_{a\text{deg}(\bar\lambda), \lambda}$ is represented by the matrix $\mathcal A_1(\lambda^{p^{a\text{deg}(\bar\lambda)-1}})\cdots\mathcal A_1(\lambda^p)\mathcal A_1(\lambda)$. And $\phi_{a\text{deg}(\bar{\lambda}),i,\lambda}$ is represented by $p^{-a\text{deg}(\bar\lambda) i(i-1)/2}\big(\wedge^i\mathcal A_1(\lambda^{p^{a\text{deg}(\bar\lambda)-1}})\big)\cdots\big(\wedge^i \mathcal A_1(\lambda)\big).$ Hence it suffices to check the lemma when replacing  $\phi_{a\text{deg}(\bar{\lambda}),i,\lambda}$ by $p^{-i(i-1)/2}\big(\wedge^i \mathcal A_1(\lambda)\big)$. Using (\ref{eqn2}) and (\ref{eqn3}), one can check this by definition. 
\end{proof}

 The following lemma gives the property of  the root of the $L$
-function $L\big(f(\bar\lambda,\mathbf x), T\big)^{(-1)^{n+1}}$.

\begin{lem}\label{lemma8} For each integer $0\le j\le n$, let $\beta_j(\bar{\lambda})=q^{-j\deg(\bar{\lambda})}\pi_j(\bar{\lambda})$. Then $\beta_j(\bar{\lambda})$ is a $p$-adic 1-unit.
\end{lem}
\begin{proof}

We have that $\beta_0(\bar{\lambda})\cdots\beta_{j}(\bar{\lambda})$ is the reciprocal characteristic root of $\phi_{a\text{deg}(\bar{\lambda}),j,\lambda}$ which is a $p$-adic 1-unit for any $0\le j\le n$. Hence $\beta_j(\bar\lambda)$ is a $p$-adic 1-unit for all $0\le j\le n$.
 \end{proof}
 
Hence for $\kappa\in \mathbb{Z}_p$, one can define the
unit root $L$-function coming from slope $j$ by
$$L_{unit}(j,\kappa, T):=\prod_{\bar{\lambda}\in |\mathbb {G}_m/\mathbb{F}_q|}\frac{1}{1-(q^{-j\deg{(\bar{\lambda})}}\pi_j(\bar{\lambda}))^{\kappa}T^{\deg(\bar{\lambda})}}.$$

\section{Symmetric power $L$-functions}
\subsection{The symmetric power}

Let $\mathcal{I}_i$ be the set of
$\{\vec{u}^{(i)}=(u_1,\cdots,u_i ):0\le  u_1<\cdots< u_i\le n\}$. Let $l_i=\binom{n+1}{i}$. Fix the lexicographical order on  $\mathcal{I}_i$ such that $\vec{u}_{0}^{(i)}<\vec{u}_{1}^{(i)}<\cdots<\vec{u}_{l_i-1}^{(i)}$ in $\mathcal{I}_i$.  
 Write $e_{\vec{u}^{(i)}}=e_{u_1}\wedge \cdots \wedge e_{u_i}$. Order the
basis of $\wedge^i\mathcal{H}_{\Lambda^{p^m}}$ by 
$$e_{\vec{u}_{0}^{(i)}}<\cdots<e_{\vec{u}_{l_i-1}^{(i)}}.$$
Note that $e_{\vec{u}_0^{(i)}}=e_0\wedge e_1\wedge \cdots \wedge e_{i-1}$.
For a positive integer $k$, define the $k$-th symmetric power ${\rm Sym}^k(\wedge^i\mathcal{H}_{\Lambda^{p^m}})$ of $\wedge^i\mathcal{H}_{\Lambda^{p^m}}$
 to be the $p$-adic finite Banach module over $\mathcal{O}_{0,p^m}$ with basis
$$\{e_{\vec{u}_{j_1}^{(i)}}\cdots e_{\vec{u}_{j_k}^{(i)}}: \vec{u}_{j_1}^{(i)},\cdots,\vec{u}_{j_k}^{(i)} \in \mathcal{I}_i\ {\rm and}\ \vec{u}_{j_1}^{(i)}\le\cdots\le \vec{u}_{j_k}^{(i)}\}.$$
Then $\phi_{m,i}$ induces a map ${\rm Sym}^k \phi_{m,i}: {\rm Sym}^k(\wedge^i\mathcal{H}_{\Lambda})\rightarrow {\rm Sym}^k(\wedge^i\mathcal{H}_{\Lambda^{p^m}})$ defined by
$${\rm Sym}^k\phi_{m,i}(e_{\vec{u}_{j_1}^{(i)}}\cdots e_{\vec{u}_{j_k}^{(i)}}):=\phi_{m,i}(e_{\vec{u}_{j_1}^{(i)}})\cdots \phi_{m,i}(e_{\vec{u}_{j_k}^{(i)}}).$$

For each $\bar{\lambda}\in \bar{\mathbb{F}}_q^*$, we denote the set $\{\pi_j(\bar{\lambda})\}_{j=0}^n$ by $\mathcal{\mu}(\bar{\lambda})$.
Let $\mathcal{L}$ be a linear algebra operation such as symmetric power, or exterior power, or tensor power, or a composition or combination of such. Let $\mathcal{L}\mathcal{\mu}(\bar{\lambda})$ be the corresponding multi-set of eigenvalues of $\mathcal{L}\bar{\alpha}_{a\deg(\bar{\lambda}),\lambda}|_{\mathcal{L}(\mathcal{H}_{\Lambda}\otimes \Omega(\lambda))}$.
Define the associated $L$-function by
$$L(\mathcal{L},T):=\prod_{\bar{\lambda}\in |\mathbb{G}_m/\mathbb{F}_q|}\prod_{\tau(\bar{\lambda})\in \mathcal{L}\mathcal{\mu}(\bar{\lambda})}(1-\tau(\bar{\lambda})T^{\deg(\bar{\lambda})})^{-1}.$$
For example, if $\mathcal{L}$ is the operator of the $i$-th exterior power, then
$$\mathcal{L}\mathcal{\mu}(\bar{\lambda})=\wedge^i \mathcal{\mu}(\bar{\lambda})=\{\pi_{u_1}(\bar{\lambda})\cdots \pi_{u_i}(\bar{\lambda}):0\le u_1<\cdots <u_i\le n\}. $$
In this case, we simply write $\pi_{u_1}(\bar{\lambda})\cdots \pi_{u_i}(\bar{\lambda})$ by $\pi_{\vec{u}^{(i)}}(\bar{\lambda})$ for $\vec{u}^{(i)}=(u_1,\cdots,u_i)\in \mathcal{I}_i$. If $\mathcal{L}$ is the operation of the $k$-th symmetric power of the $i$-th exterior power, then
$$\mathcal{L}\mathcal{\mu}(\bar{\lambda})={\rm Sym}^k(\wedge^i \mathcal{\mu}(\bar{\lambda}))=\{\pi_{\vec{u}_{j_1}^{(i)}}(\bar{\lambda})\cdots \pi_{\vec{u}_{j_k}^{(i)}}(\bar{\lambda}):\vec{u}_{j_1}^{(i)},\cdots, \vec{u}_{j_k}^{(i)}\in \mathcal{I}_i,\vec{u}_{j_1}^{(i)}\le\cdots \le \vec{u}_{j_k}^{(i)}\}.$$

Let $k_s$ and $k_s'$ be two sequences of positive integers such that
$\lim_{s\rightarrow \infty} k_s=\lim_{s\rightarrow \infty} k_s'=\infty$ as integers, and $\lim_{s\rightarrow \infty}k_s'=-\kappa$,
$\lim_{s\rightarrow \infty}k_s=\kappa$ $p$-adically.
The following lemma follows from \cite[Lemma 7.2 ]{Wan99}.

\begin{lem}\label{lemma6}
 Let $\kappa\in \mathbb{Z}_p$. Then the following limiting formulas hold
\begin{equation}\label{eqn12}
L_{unit}(j,\kappa, T)=\lim_{s\rightarrow \infty} L(\phi_{a,j}^{k_s'} \otimes \phi_{a,j+1}^{k_s},T )
\end{equation}
for $j\ge 1$, and
$$L_{unit}(0,\kappa, T)=\lim_{s\rightarrow \infty} L(\phi_{a,1}^{k_s},T ).$$
\end{lem}

Note that $\wedge^l \phi_{a,j}\rightarrow 0$ as $l\rightarrow \infty$. Using \cite[Lemma 4.2]{Wan99} and applying Lemma \ref{lemma6} to the proof of the last equation in \cite[Section 6]{Wan99}, we have the following formulas hold: for $j\ge 1$,
\begin{align}\label{eqn06}
L_{unit}(j,\kappa, T)=&\lim_{s\rightarrow \infty}  L({\rm Sym}^{k_s'} \phi_{a,j} \otimes {\rm Sym}^{k_s}\phi_{a,j+1},T ) \\\nonumber
&\cdot \prod_{\iota_1\ge 2} \lim_{s\rightarrow \infty} L(\mathcal{E}(\iota_1)\otimes {\rm Sym}^{k_{s}}\phi_{a,j+1},T )^{(-1)^{\iota_1-1}(\iota_1-1)}\\\nonumber
&\cdot \prod_{\iota_2\ge 2} \lim_{s\rightarrow \infty}L({\rm Sym}^{k_s'} \phi_{a,j} \otimes \mathcal{E}(\iota_2),T )^{(-1)^{\iota_2-1}(\iota_2-1)}\\\nonumber
& \cdot\prod_{\iota_1,\iota_2\ge 2} \lim_{s\rightarrow \infty}L(\mathcal{E}(\iota_1,\iota_2),T )^{(-1)^{\iota_1+\iota_2}(\iota_1+\iota_2-2)},
\end{align}
where
$$\mathcal{E}(\iota_1)= {\rm Sym}^{k_s'-\iota_1} \phi_{a,j} \otimes \wedge^{\iota_1} \phi_{a,j},\ \mathcal{E}(\iota_2)={\rm Sym}^{k_{s}-\iota_2}\phi_{a,j+1}\otimes \wedge^{\iota_2} \phi_{a,j+1}$$
and
$$\mathcal{E}(\iota_1,\iota_2)={\rm Sym}^{k_s'-\iota_1} \phi_{a,j} \otimes \wedge^{\iota_1} \phi_{a,j}\otimes {\rm Sym}^{k_{s}-\iota_2}\phi_{a,j+1}\otimes \wedge^{\iota_2} \phi_{a,j+1}, $$
and
\begin{align}\label{eqn0002}
L_{unit}(0,\kappa, T)=\lim_{s\rightarrow \infty}  L( {\rm Sym}^{k_s}\phi_{a,1},T )
 \prod_{\iota\ge 2} \lim_{s\rightarrow \infty} L({\rm Sym}^{k_s-\iota} \phi_{a,1} \otimes \wedge^{\iota} \phi_{a,1},T )^{(-1)^{\iota-1}(\iota-1)}.
\end{align}
{\bf Remark}: From (\ref{eqn06}) and (\ref{eqn0002}) we conclude that the unit root  of
$L_{unit}(j,\kappa, T)(j\ge 0)$ only comes from the first factor. 

\subsection{Infinity symmetric power}
For $1\le i\le n+1$ and $m\in \mathbb{Z}_{\ge0}$, let $S^{(i)}(\mathcal{O}_{0,p^m})$  be the formal power series ring over
$\mathcal{O}_{0,p^m}$ in the variables $\{ g_{\vec{u}_j^{(i)}}\}$ which are indexed by $\mathcal{I}_i\setminus\{\vec{u}_{0}^{(i)}\}$. Equip this ring with the sup-norm on the coefficients in $\mathcal{O}_{0,p^m}$.
Write the monomial of degree or length $r$ in the variables $\{ g_{\vec{u}_j^{(i)}}\}$ by $g_{\vec{\mathbf{u}}^{(i)}}=g_{\vec{u}_{j_1}^{(i)}}\cdots g_{\vec{u}_{j_r}^{(i)}}$ and let $|\vec{\mathbf{u}}^{(i)}|=r$. For $\zeta\in S^{(i)}(\mathcal{O}_{0,p^m})$, define length$(\zeta)$ by the supremum of the length of those monomials appearing with nonzero coefficients.
Denote by $\mathcal{M}^{(i)}$ the set of all indices $\vec{\mathbf{u}}^{(i)}$ corresponding to monomials $g_{\vec{\mathbf{u}}^{(i)}}$. Define a $p$-adic Banach subspace of $S^{(i)}(\mathcal{O}_{0,p^m})$ over $\mathcal{O}_{0,p^m}$ by
$$\mathcal{S}^{(i)}(\mathcal{O}_{0,p^m}):=\{\sum_{\vec{\mathbf{u}}^{(i)}\in \mathcal{M}^{(i)}} \zeta(\vec{\mathbf{u}}^{(i)})g_{\vec{\mathbf{u}}^{(i)}}: \zeta(\vec{\mathbf{u}}^{(i)})\in \mathcal{O}_{0,p^m},\zeta(\vec{\mathbf{u}}^{(i)})\rightarrow 0\ {\rm as}\ |\vec{\mathbf{u}}^{(i)}|\rightarrow \infty\}.$$

We view $\wedge^i\mathcal{H}_{\Lambda^{p^m}}$ as a subspace of $\mathcal{S}^{(i)}(\mathcal{O}_{0,p^m})$ by defining $\Upsilon:\wedge^i\mathcal{H}_{\Lambda^{p^m}}\to\mathcal{S}^{(i)}(\mathcal{O}_{0,p^m})$ via
$$\Upsilon(e_{\vec{u}_0^{(i)}})=1 \ {\rm and}\  \Upsilon(e_{\vec{u}_j^{(i)}})=g_{\vec{u}_j^{(i)}}\ {\rm for}\ j\ge 1. $$
From (\ref{eqnaij}), we conclude that $\Upsilon\circ\phi_{m,i}(e_{\vec{u}_0^{(i)}})=1+\eta$ with $\eta\in \mathcal{S}^{(i)}(\mathcal{O}_{0,p^m})$ and $|\eta|<1$. Thus $\big(\Upsilon\circ\phi_{m,i}(e_{\vec{u}_0^{(i)}})\big)^{\kappa}$ is well-defined and belongs to $\mathcal{S}^{(i)}(\mathcal{O}_{0,p^m})$
for $\kappa\in \mathbb{Z}_p.$ Hence we can define a map from $\mathcal{S}^{(i)}(\mathcal{O}_{0})$ to $\mathcal{S}^{(i)}(\mathcal{O}_{0,p^m})$ by
$$[\phi_{m,i}]_{\infty,\kappa}(g_{\vec{u}_{j_1}^{(i)}}\cdots g_{\vec{u}_{j_r}^{(i)} } ):=\big(\Upsilon\circ \phi_{m,i}(e_{\vec{u}_0^{(i)}})\big)^{\kappa-r}\big(\Upsilon\circ\phi_{m,i}(e_{\vec{u}_{j_1}^{(i)}})\big)\cdots \big(\Upsilon\circ\phi_{m,i}(e_{\vec{u}_{j_r}^{(i)}})\big).$$

Let $k$ be a positive integer. Then define an $\mathcal{O}_{0,p^m}$-submodule of $\mathcal{S}^{(i)}(\mathcal{O}_{0,p^m})$ by
$$\mathcal{S}_k^{(i)}(\mathcal{O}_{0,p^m}):=\{ \zeta \in \mathcal{S}^{(i)}(\mathcal{O}_{0,p^m}): {\rm length}(\zeta)\le k\}.$$
There is an identification ${\rm Sym}^k(\wedge^i\mathcal{H}_{\Lambda^{p^m}})\cong\mathcal{S}_k^{(i)}(\mathcal{O}_{0,p^m})$ by sending 
$$e_{\vec{u}_0^{(i)}}^{k-r}e_{\vec{u}_{j_1}^{(i)}}\cdots e_{\vec{u}_{j_r}^{(i)}}\mapsto g_{\vec{u}_{j_1}^{(i)}}\cdots g_{\vec{u}_{j_r}^{(i)}}. $$
Define a map $[\phi_{m,i}]_k$ from  $\mathcal{S}_k^{(i)}(\mathcal{O}_{0})$ to $\mathcal{S}_k^{(i)}(\mathcal{O}_{0,p^m})$ by
$$[\phi_{m,i}]_k(g_{\vec{u}_{j_1}^{(i)}}\cdots g_{\vec{u}_{j_r}^{(i)}} ):=\big(\Upsilon\circ\phi_{m,i}(e_{\vec{u}_0^{(i)}})\big)^{k-r}\big(\Upsilon\circ\phi_{m,i}(e_{\vec{u}_{j_1}^{(i)}})\big)\cdots \big(\Upsilon\circ\phi_{m,i}(e_{\vec{u}_{j_r}^{(i)}})\big) .$$
Hence
$$ [\phi_{m,i}]_k(g_{\vec{u}_{j_1}^{(i)}}\cdots g_{\vec{u}_{j_r}^{(i)}} )\cong {\rm Sym}^k\phi_{m,i}\big((e_{\vec{u}_0^{(i)}})^{k-r}e_{\vec{u}_{j_1}^{(i)}}\cdots e_{\vec{u}_{j_r}^{(i)}}\big). $$

We extend the map $[\phi_{m,i}]_k$ to $\mathcal{S}^{(i)}(\mathcal{O}_{0})$ by requiring
$$[\phi_{m,i}]_{k}(g_{\vec{u}_{j_1}^{(i)}}\cdots g_{\vec{u}_{j_r}^{(i)}} )=0$$
for $\vec{u}_0^{(i)}< \vec{u}_{j_1}^{(i)}\le \cdots\le \vec{u}_{j_r}^{(i)}$ with $r>k$.
Using the same argument as \cite[Lemma 4.3]{HS17}, we have
\begin{lem} \label{lemma7}
Let $\kappa\in \mathbb{Z}_p$ and $k_s$ be a sequence of positive integers tending to infinity as integers and $\lim_{s\rightarrow \infty}k_s=\kappa$
$p$-adically. Then as a map from $\mathcal{S}^{(i)}(\mathcal{O}_{0})$ to $\mathcal{S}^{(i)}(\mathcal{O}_{0,p^m})$, we have
$$ \lim_{s\rightarrow \infty}[\phi_{m,i}]_{k_s}= [\phi_{m,i}]_{\infty,\kappa}.$$
\end{lem}

 Define
$$L([\phi_{a,i}]_{\infty,\kappa},T):=\prod_{\bar{\lambda}\in |\mathbb{G}_m/\mathbb{F}_q|}\frac{1}{\det(1-T[\phi_{a\deg(\bar{\lambda}),i,\lambda}]_{\infty,\kappa}|_{\mathcal{S}^{(i)}(\mathcal{O}_{0})\otimes \Omega(\lambda)})},$$
where $[\phi_{a\deg(\bar{\lambda}),i,\lambda}]_{\infty,\kappa}$ is the specialization of $[\phi_{a\deg(\bar{\lambda}),i}]_{\infty,\kappa}$ at $\Lambda=\lambda$. Similarly, we can define the $L$-function $L([\phi_{a,i}]_{k},T)$ for any integer $k$.
The same arguments as \cite[Corollary 2.4]{H14} leads to
\begin{equation}\label{eqn0003} \lim_{s\rightarrow \infty}L({\rm Sym}^{k_s}\phi_{a,i},T)=\lim_{s\rightarrow \infty}L([\phi_{a,i}]_{k_s},T)= L([\phi_{a,i}]_{\infty,\kappa},T).
\end{equation}

Define $\psi_{\Lambda}:\mathcal{S}^{(i)}(\mathcal{O}_{0,p})\to\mathcal{S}^{(i)}(\mathcal{O}_{0})$ by 
$$\psi_{\Lambda}: \sum a(s;\vec{\mathbf{u}}^{(i)})\Lambda^sg_{\vec{\mathbf{u}}^{(i)}}\mapsto \sum a(ps;\vec{\mathbf{u}}^{(i)})\Lambda^s g_{\vec{\mathbf{u}}^{(i)}}.$$
One checks that $$\rho^{(i,\infty,\kappa)}:=\psi_{\Lambda}^a\circ [\phi_{a,i}]_{\infty,\kappa}:\mathcal{S}^{(i)}(\mathcal{O}_{0})\to\mathcal{S}^{(i)}(\mathcal{O}_{0})$$
is a completely continuous operator as in \cite[section 3]{HS17}.

For $1\le j\le n$, let $\mathcal{S}^{(j)}(\mathcal{O}_{0}) \hat{\otimes} \mathcal{S}^{(j+1)}(\mathcal{O}_{0})$ be the completed tensor product of $\mathcal{S}^{(j)}(\mathcal{O}_{0}) $ and $\mathcal{S}^{(j+1)}(\mathcal{O}_{0})$. Let  $\beta^{(j,\infty,\kappa)}:=\rho^{(j,\infty,-\kappa)}\otimes \rho^{(j+1,\infty,\kappa)}$, which is also nuclear.
Let $B^{(j,\infty,\kappa)}$ be the matrix of $[\phi_{a,j}]_{\infty,-\kappa}\otimes[\phi_{a,j+1}]_{\infty,\kappa}$ with respect to the basis
$$\{g_{\vec{\mathbf{u}}^{(j)}} \otimes g_{\vec{\mathbf{v}}^{(j+1)}}: \vec{\mathbf{u}}^{(j)}\in \mathcal{M}^{(j)}, \vec{\mathbf{v}}^{(j+1)}\in \mathcal{M}^{(j+1)} \} .$$
Then $$B^{(j,\infty,\kappa)}=\sum_{r\ge0}\pi^{(n+1)r}b_r^{(j,\infty,\kappa)}\Lambda^r,$$
where the coefficient $b_r^{(j,\infty,\kappa)}$ is a matrix with entries in $R$.
Let $F_{B^{(j,\infty,\kappa)} }:=(b_{qr-s}^{(j,\infty,\kappa)} )_{r,s\in \mathbb{Z}_{\ge 0}}$ with $b_{qr-s}^{(j,\infty,\kappa)}:=0 $ if $qr-s<0$.
We may view $\mathcal{S}^{(j)}(\mathcal{O}_{0}) \hat{\otimes} \mathcal{S}^{(j+1)}(\mathcal{O}_{0})$ as a $p$-adic Banach space over $R$ with orthonormal basis
$$\mathcal{B}=\{\pi^{(n+1)r}\Lambda^r g_{\vec{\mathbf{u}}^{(j)}} \otimes g_{\vec{\mathbf{v}}^{(j+1)}} :r\in \mathbb{Z}_{\ge 0}, \vec{\mathbf{u}}^{(j)}\in \mathcal{M}^{(j)}, \vec{\mathbf{v}}^{(j+1)}\in \mathcal{M}^{(j+1)} \}.$$
As described prior to \cite[Lemma 2.3]{HS14}, $F_{B^{(j,\infty,\kappa)} }$ is the matrix of $\beta^{(j,\infty,\kappa)}$ with respect to the basis $\mathcal{B}$.
Using Dwork's trace formula \cite[lemma 4.1]{Wan96} and the same argument as \cite[Equation (10)]{HS17},
 we have that
\begin{align*}
L([\phi_{a,j}]_{\infty,-\kappa}\otimes[\phi_{a,j+1}]_{\infty,\kappa},T)\nonumber
= \frac{\det(1-\beta^{(j,\infty,\kappa)}T)}{\det(1-q\beta^{(j,\infty,\kappa)}T)}.
\end{align*}
It follows from (\ref{eqn0003}) that
\begin{equation}\label{eqn05}\lim_{s\rightarrow \infty}L({\rm Sym}^{k_s'}\phi_{a,j}\otimes {\rm Sym}^{k_s}\phi_{a,j+1},T)=\frac{\det(1-\beta^{(j,\infty,\kappa)}T)}{\det(1-q\beta^{(j,\infty,\kappa)}T)}.
\end{equation}

Similarly, when $j=0$, let $\beta^{(0,\infty,\kappa)}=\rho^{(1,\infty,\kappa)}$. We have the following formula
\begin{equation}
\lim_{s\rightarrow \infty}  L( {\rm Sym}^{k_s}\phi_{a,1},T )=\frac{\det(1-\beta^{(0,\infty,\kappa)}T)}{\det(1-q\beta^{(0,\infty,\kappa)}T)}.
\end{equation}
The factor $\det(1-q\beta^{(j,\infty,\kappa)}T)$ would not contribute any zero or pole in the unit disc. From (\ref{eqnaij}) and (\ref{eqn1}), one checks that the matrix $B^{(j,\infty,\kappa)}$ takes the form $\begin{pmatrix} 1 & 0 \\ 0 & 0 \end{pmatrix}$ mod $\pi$. Thus $\det(1-TF_{B^{(j,\infty,\kappa)}})\equiv 1-T$ mod $\pi$. Combining with (\ref{eqn06}) and (\ref{eqn0002}),  we prove that:
  \begin{thm}\label{thm2} For $0\le j\le n$, the unit root $L$-function $L_{unit}(j,\kappa, T)$ has a unique unit root, which is the unique unit root of $\det(1-\beta^{(j,\infty,\kappa)}T)$.
  \end{thm}

 \section{Dual theory}
\subsection{The dual module}
 Let
 $$\mathcal{C}_{0,p^{m}}^*:=\{\zeta^*=\sum_{{\bf u}\in \mathbb{Z}^n} \zeta({\bf u}) \pi^{-w({\bf u})} \Lambda^{-p^m s({\bf u})}\mathbf{x}^{-{\bf u}}: \zeta({\bf u})\in \mathcal{O}_{0,p^{m}} \}.$$
We define a pairing
\begin{equation}\label{pair} \langle ,\rangle: \mathcal{C}_{0,p^{m}}\times \mathcal{C}_{0,p^{m}}^*\rightarrow \mathcal{O}_{0,p^{m}}
\end{equation}
by
$$\langle \zeta,\zeta^*\rangle:={\rm the\ constant\ term\ with\ respect\ to\ } \mathbf{x} {\rm\ of\ the\ product}\ \zeta\cdot\zeta^*$$
for $\zeta=\sum_{{\bf u}\in \mathbb{Z}^n} \zeta_1({\bf u}) \pi^{w({\bf u})} \Lambda^{p^m s({\bf u})}\mathbf{x}^{{\bf u}}$ and $\zeta^*=\sum_{{\bf v}\in \mathbb{Z}^n} \zeta_2({\bf v}) \pi^{-w({\bf v})} \Lambda^{-p^m s({\bf v})}\mathbf{x}^{-{\bf v}}$.
This pairing is well-defined since $\zeta_1({\bf u})\rightarrow 0$ as $w({\bf u})\rightarrow \infty$ and $\zeta_2({\bf v})$ is bounded.  This is a perfect pairing as shown in \cite[Section 2]{SP0}. Hence $\mathcal{C}_{0,p^{m}}^*$ is the dual $\mathcal{O}_{0,p^{m}}$-module of $\mathcal{C}_{0,p^{m}}$.

Let $$D_{i,\Lambda^{p^m}}^*:=-x_i\frac{\partial}{\partial x_i}+\pi (x_i-\frac{\Lambda^{p^m}}{x_1\cdots x_n})$$
acting on $\mathcal{C}_{0,p^m}$
 for $i=1,\cdots,n$.
 Then for $\zeta^*\in \mathcal{C}_{0,p^{m}}^*$ and $\zeta\in \mathcal{C}_{0,p^{m}}$, we have that (see \cite[(2.3.13)]{SP0})

\begin{equation}\label{eqD}
\langle \zeta,D_{i,\Lambda^{p^m}}^*\zeta^*\rangle =\langle D_{i,\Lambda^{p^m}}\zeta,\zeta^*\rangle.
\end{equation}
 Recall that $\mathcal{H}_{\Lambda^{p^m}}=\mathcal{C}_{0,p^m}/\sum_{i=1}^n D_{i,\Lambda^{p^m}} \mathcal{C}_{0,p^{m}}$.
 The $\mathcal{O}_{0,p^{m}}$-dual of $\mathcal{H}_{\Lambda^{p^m}}$ is the annihilator  of $\sum_{i=1}^n D_{i,\Lambda^{p^m}} \mathcal{C}_{0,p^{m}}$ in $\mathcal{C}_{0,p^{m}}^*$, denote it by $\mathcal{H}_{\Lambda^{p^m}}^*$. Then by (\ref{eqD}) we conclude that
 $$\mathcal{H}_{\Lambda^{p^m}}^*=\{\zeta^*\in \mathcal{C}_{0,p^{m}}^*: D_{i,\Lambda^{p^m}}^* \zeta^*=0\ {\rm for\ all}\ 1\le i\le n \} .$$
 Then $\mathcal{H}_{\Lambda^{p^m}}^*$ is also a free $\mathcal{O}_{0,p^m}$-module of rank $n+1$. Denote the dual basis of $\{e_i\}_{i=0}^n$ by  $\{e_{i,\Lambda^{p^m}}^*\}_{i=0}^n $. Note that this basis relies on $m$.

For any $m\in \mathbb{Z}_{\ge 0}$ and ring $\mathcal{R}$ such that $\mathcal{O}_{0,p^m}\subset \mathcal{R}\subset \Omega[[\Lambda]]$, we define
 $$\mathcal{H}_{\Lambda^{p^m}}(\mathcal{R}):=\mathcal{H}_{\Lambda^{p^m}}\otimes_{\mathcal{O}_{0,p^m}} \mathcal{R},$$
 and its dual $\mathcal R$-module by
 $$ \mathcal{H}_{\Lambda^{p^m}}^*(\mathcal{R}):=\mathcal{H}_{\Lambda^{p^m}}^*\otimes_{\mathcal{O}_{0,p^m}} \mathcal{R}.$$

 Let $\Phi_{\mathbf{x}}$ act on the
 formal power series ring involving $\mathbf{x}$ by $\Phi_{\mathbf{x}}(\sum a_{\mathbf{u}}\mathbf{x}^{\mathbf{u}}):=\sum a_{\mathbf{u}}\mathbf{x}^{p\mathbf{u}}$ and $\alpha_{1}^*:=F(\Lambda,\mathbf{x})\circ \Phi_{\mathbf{x}}$. 
 It follows from the remark of \cite[Proposition 2.4.4]{SP0} and \cite[Proposition 2.4.7]{SP0} that we have the following property about $\alpha_{1}^*$.
 \begin{lem} \label{lem4.1}
We have that $\alpha_{1}^*$ is an $\mathcal{O}_{0,p}$-linear map from $\mathcal{H}_{\Lambda^p}^{*}(\mathcal{O}_{0,p})$
to $\mathcal{H}_{\Lambda}^{*}(\mathcal{O}_{0,p})$ and
\begin{align*}
\langle \zeta, \alpha_{1}^*\zeta^*\rangle=\langle\bar{\alpha}_1\zeta, \zeta^*\rangle
\end{align*}
for $\zeta\in \mathcal{H}_{\Lambda}(\mathcal{O}_{0})$ and $\zeta^*\in \mathcal{H}_{\Lambda^p}^{*}(\mathcal{O}_{0,p})$.
 \end{lem}
{\bf Remark}:
 When considering the action of $\alpha_{1}^*$ on $\mathcal{H}_{\Lambda^p}^{*}(\mathcal{O}_{0,p})$, we regard both image and preimage as subspaces of the power series ring involving $\Lambda$ and $\mathbf x$. And the pairing takes values in $\mathcal{O}_{0,p}$.

From Lemma \ref{lem4.1} we conclude that the matrix of $\alpha_{1}^*$ with respect to the basis $\{e_{i,\Lambda^{p}}^*\}_{i=0}^n $
  and $\{e_{i,\Lambda}^*\}_{i=0}^n $ is the transpose the matrix of
$\bar{\alpha}_1$ acting on $\{ e_i \}_{i=0}^n$.
That is,
\begin{equation}\label{eqnfro}
\alpha_{1}^*(e_{0,\Lambda^{p}}^*,\ldots,e_{n,\Lambda^{p}}^*)=(e_{0,\Lambda}^*,\ldots,e_{n,\Lambda}^*)\mathcal{A}_1^{T}.
\end{equation}

For $m\ge 1$, define $\alpha_{m}^*:=F_m(\Lambda,\mathbf{x})\circ \Phi_{\mathbf{x}}^m $. 
Similarly, we have that
 \begin{align}\label{eqn9}
\langle\zeta, \alpha_{m}^*\zeta^*\rangle=\langle\bar{\alpha}_m\zeta,\zeta^*\rangle
\end{align}
for $\zeta\in \mathcal{H}_{\Lambda}(\mathcal{O}_0)$ and $\zeta^*\in \mathcal{H}_{\Lambda^{p^m}}^{*}(\mathcal{O}_{0,p^m})$.
Hence 
\begin{equation}\label{eqnm}
\alpha_{m}^*(e_{0,\Lambda^{p ^m}}^*,\ldots,e_{n,\Lambda^{p^m}}^*)=(e_{0,\Lambda}^*,\ldots,e_{n,\Lambda}^*)\mathcal{A}_m^{T}(\Lambda).
\end{equation}

Define
 $\wedge^i\mathcal{H}_{\Lambda^{p^m}}(\mathcal{R}):=\wedge^i \mathcal{H}_{\Lambda^{p^m}}\otimes_{\mathcal{O}_{0,p^m}} \mathcal{R}$.  Let $\wedge^i H_{\Lambda^{p^m}}^{*}(\mathcal{R})$ be the free $\mathcal{R}$-module with basis
$$\{e_{\vec{u}^{(i)},\Lambda^{p^m}}^*:=e_{u_1,\Lambda^{p^m}}^*\wedge \cdots \wedge e_{u_i,\Lambda^{p^m}}^*\}_{\vec{u}^{(i)}=( u_1,\cdots,u_i)\in \mathcal{I}_i}.$$
Define the map  $\phi_{m,i}^*: \wedge^i H_{\Lambda^{p^m}}^{*}(\mathcal{O}_{0,p^m}) \rightarrow \wedge^i H_{\Lambda}^{*}(\mathcal{O}_{0,p^m})$ by
$$\phi_{m,i}^*(e_{u_1,\Lambda^{p^m}}^*\wedge \cdots \wedge e_{u_i,\Lambda^{p^m}}^*)=p^{-mi(i-1)/2} \alpha_{m}^*(e_{u_1,\Lambda^{p^m}}^*) \wedge \cdots \wedge \alpha_{m}^*(e_{u_i,\Lambda^{p^m}}^*).$$

\begin{lem}\label{lemma2}
Let $(B_{\vec{u}^{(i)},\vec{v}^{(i)}}^*)_{\vec{u}^{(i)},\vec{v}^{(i)}\in \mathcal{I}_i}$ be the matrix of $\phi_{1,i}^*$ with respect to $\{e_{\vec{v}^{(i)},\Lambda^p}^*\}_{\vec{v}^{(i)}\in \mathcal{I}_i}$ and
$\{e_{\vec{u}^{(i)},\Lambda}^*\}_{\vec{u}^{(i)}\in \mathcal{I}_i}$
Then
$B_{\vec{u}^{(i)},\vec{u}^{(i)}}^{*}\equiv p^{u_1+\cdots+u_i-i(i-1)/2}\mod \Lambda$ and $B_{\vec{u}^{(i)},\vec{v}^{(i)}}^{*}\equiv 0 \mod \Lambda$ for $\vec{u}^{(i)}>\vec{v}^{(i)}$.
\end{lem}
\begin{proof}
Let $\vec{v}^{(i)}=(v_1,\cdots,v_i)$. From (\ref{eqnfro}), we have that
\begin{align*}
\phi_{1,i}^*(e_{\vec{v}^{(i)},\Lambda^p}^{*})=&p^{-i(i-1)/2} \alpha_{1}^*(e_{v_1, \Lambda^p}^*)\wedge \cdots \wedge \alpha_{1}^*(e_{v_i, \Lambda^p}^*)\\
=&p^{-i(i-1)/2}(a_{v_1,0}e_{0, \Lambda}^*+\cdots+a_{v_1,n}e_{n, \Lambda}^*)\wedge \cdots \wedge (a_{v_i,0}e_{0, \Lambda}^*+\cdots+a_{v_i,n}e_{n, \Lambda}^*)\\
=&p^{-i(i-1)/2} \sum_{0\le u_1<\cdots<u_i\le n} \sum_{\sigma\in S_i}(-1)^{{\rm sign(\sigma)}}a_{v_1,u_{\sigma (1)}}\cdots a_{v_i,u_{\sigma (i)}}e_{u_1, \Lambda}^*\wedge\cdots \wedge e_{u_i, \Lambda}^*,
\end{align*}
where $S_i$ is the symmetric group on $i$ letters.
Then
$$B_{\vec{u}^{(i)},\vec{v}^{(i)}}^{*}=p^{-i(i-1)/2}\sum_{\sigma\in S_i}(-1)^{{\rm sign(\sigma)}}a_{v_1,u_{\sigma (1)}}\cdots a_{v_i,u_{\sigma (i)}}.$$
When $\vec{v}^{(i)}=\vec{u}^{(i)}$, 
by (\ref{eqnaij}) and (\ref{eqn1})
$$B_{\vec{u}^{(i)},\vec{u}^{(i)}}^{*}\equiv p^{-i(i-1)/2}a_{u_1,u_1}\cdots a_{u_i,u_i}\equiv p^{u_1+\cdots+u_i-i(i-1)/2}\mod \Lambda .$$
When $\vec{v}^{(i)}<\vec{u}^{(i)}$, there is an integer $1\le k\le i$ such that $v_k<u_k$.
It follows from (\ref{eqn1}) that $a_{v_1,u_1}\cdots a_{v_i,u_i}\equiv 0 \mod \Lambda$.
This finishes the proof of Lemma \ref{lemma2}.
\end{proof}

From (\ref{eqn005}) and (\ref{eqnm}), we have that
\begin{cor}\label{coro}Let $(B_{\vec{u}^{(i)},\vec{v}^{(i)}}^{*(m)})_{\vec{u}^{(i)},\vec{v}^{(i)}\in \mathcal{I}_i}$ be the matrix of the map $\phi_{m,i}^*$ with respect to $\{e_{\vec{v}^{(i)},\Lambda^{p^m}}^*\}_{\vec{v}^{(i)}\in \mathcal{I}_i}$ and
$\{e_{\vec{u}^{(i)},\Lambda}^*\}_{\vec{u}^{(i)}\in \mathcal{I}_i}$.
Then $B_{\vec{u}^{(i)},\vec{u}^{(i)}}^{*(m)}\equiv p^{m(u_1+\cdots+u_i-i(i-1)/2)}\mod \Lambda$ and $B_{\vec{u}^{(i)},\vec{v}^{(i)}}^{*(m)}\equiv 0 \mod \Lambda$ for $\vec{u}^{(i)}>\vec{v}^{(i)}$.
\end{cor}

Suppose that $s\ge r\ge m$. For $\zeta=\zeta_{1}\wedge\cdots\wedge \zeta_{i}\in \wedge^i\mathcal{H}_{\Lambda^{p^m}}(\mathcal{O}_{0,p^r})$ and
$\zeta^*=\zeta_{1}^*\wedge\cdots\wedge \zeta_{i}^*\in \wedge^i\mathcal{H}_{\Lambda^{p^m}}^*(\mathcal{O}_{0,p^s})$, define
\begin{equation}\label{eqn8}
\langle \zeta_{1}\wedge \cdots \wedge \zeta_{i}, \zeta_{1}^*\wedge \cdots \wedge \zeta_{i}^*\rangle:=\sum_{\sigma\in S_i} {\rm sign} (\sigma)\prod_{s=1}^i \langle \zeta_{s},\zeta_{\sigma(s)}^*\rangle.
\end{equation}
This pairing takes values in $\mathcal{O}_{0,p^s}$. 
 In particular, from (\ref{eqn9}) and (\ref{eqn8}) we obtain that for $\zeta^*\in \wedge^i\mathcal{H}_{\Lambda^q}^*(\mathcal{O}_{0,q})$ and $\zeta\in \mathcal{H}_{\Lambda}^i(\mathcal{O}_{0})$,
\begin{equation}\label{eqn0001}
\langle \zeta, \phi_{a,i}^*(\zeta^*)\rangle=\langle \phi_{a,i}(\zeta), \zeta^*\rangle.
\end{equation}

As
 ${\rm Sym}^k(\wedge^i\mathcal{H}_{\Lambda^{p^m}}(\mathcal{R}))\cong{\rm Sym}^k(\wedge^i\mathcal{H}_{\Lambda^{p^m}})\otimes_{\mathcal{O}_{0,p^m}} \mathcal{R}$, we have that ${\rm Sym}^k (\wedge^i\mathcal{H}_{\Lambda^{p^m}}^*(\mathcal{R}))$ is a free $\mathcal{R}$-module with basis
$$ \{e_{\vec{u}_{j_1}^{(i)}, \Lambda^{p^m}}^{*}\cdots e_{\vec{u}_{j_k}^{(i)}, \Lambda^{p^m}}^{*}: \vec{u}_{j_1}^{(i)},\cdots,\vec{u}_{j_k}^{(i)}\in \mathcal{I}_i, \vec{u}_{j_1}^{(i)}\le \cdots \le\vec{u}_{j_k}^{(i)}\}.$$

Now we extend the pairing above to the symmetric power spaces as follows.
For $\zeta=\zeta_1\cdots \zeta_k\in {\rm Sym}^k(\wedge^i\mathcal{H}_{\Lambda^{p^m}}(\mathcal{O}_{0,p^r}))$ and $\zeta^*=\zeta_1^*\cdots \zeta_k^* \in {\rm Sym}^k (\wedge^i\mathcal{H}_{\Lambda^{p^m}}^*(\mathcal{O}_{0,p^s}))$,
 define
\begin{equation*}
\langle\zeta,\zeta^*\rangle_k:=\langle\zeta_1\cdots \zeta_k,\zeta_1^*\cdots \zeta_k^*\rangle_k=\frac{1}{k!}\sum_{\sigma\in S_k}\prod_{i=1}^k\langle\zeta_i,\zeta_{\sigma (i)}^*\rangle.
\end{equation*}

Define a map ${\rm Sym}^k \phi_{m,i}^*:{\rm Sym}^k (\wedge^i\mathcal{H}_{\Lambda^{p^m}}^*(\mathcal{O}_{0,p^m}))\to{\rm Sym}^k( \wedge^i\mathcal{H}_{\Lambda}^*(\mathcal{O}_{0,p^m}))$ by
$${\rm Sym}^k \phi_{m,i}^*( e_{\vec{u}_{j_1}^{(i)}, \Lambda^{p^m}}^{*}\cdots e_{\vec{u}_{j_k}^{(i)}, \Lambda^{p^m}}^{*})= \phi_{m,i}^*( e_{\vec{u}_{j_1}^{(i)}, \Lambda^{p^m}}^{*})\cdots \phi_{m,i}^*( e_{\vec{u}_{j_k}^{(i)}, \Lambda^{p^m}}^{*}).$$
When $m=a$, it follows from (\ref{eqn0001}) that
\begin{equation}\label{eqn003}
\langle \zeta, {\rm Sym}^k\phi_{a,i}^*(\zeta^*) \rangle_k=\langle {\rm Sym}^k\phi_{a,i}(\zeta), \zeta^*\rangle_k
\end{equation}
 for $\zeta^*\in {\rm Sym}^k (\wedge^i\mathcal{H}_{\Lambda^{q}}^* (\mathcal{O}_{0,q}))$ and $\zeta\in {\rm Sym}^k(\wedge^i\mathcal{H}_{\Lambda}^*(\mathcal{O}_{0}))$.

\subsection{The estimation}
For $m\geq 0$, define an $R$-module
$$\mathcal{O}^*_{0,p^{m}}:=\{ \sum_{r\in \mathbb{Z}_{\ge 0}} a^*(r)\pi^{\frac{-(n+1)r}{p^{m}}}\Lambda^{-r}: a^*(r)\in R\}.$$
Denote it by $\mathcal{O}_0^*$ for $m=0$.
Then $\mathcal{O}_{0,p^m}^*\subset \mathcal{O}_{0}^*$.
For positive integer $k$, define
$$\mathcal{M}_{k}^{(i)}=\{ \vec{\mathbf {v}}^{(k)}:=\vec{v}_{j_1}^{(i)}\cdots \vec{v}_{j_k}^{(i)}: \vec{v}_{j_1}^{(i)}\le \cdots\le\vec{v}_{j_k}^{(i)} \text{ in } \mathcal{I}_i\}\ {\rm and\ write}\ \mathbf{e}_{\vec{{\bf v}}^{(k)},\Lambda}^*=e_{\vec{v}_{j_1}^{(i)},\Lambda}^{*}\cdots e_{\vec{v}_{j_k}^{(i)},\Lambda}^{*}.$$
Let
$$S_{k,i,\Lambda}^*(\mathcal{O}^*_{0,p^{m}}):=\Big\{\sum_{r\in \mathbb{Z}_{\ge 0},\vec{{\bf v}}^{(k)}\in \mathcal{M}_{k}^{(i)}}a^*(r,\vec{{\bf v}}^{(k)})\pi^{\frac{-(n+1)r}{p^{m}}}\Lambda^{-r}\mathbf{e}_{\vec{{\bf v}}^{(k)},\Lambda}^*: a^*(r,\vec{{\bf v}}^{(k)})\in R\Big\}. $$

Let $\Psi_{\Lambda}:\Lambda \mapsto \Lambda^p$. We define $$\rho^{*(i,k)}:= {\rm Pr}_{\Lambda}\circ {\rm Sym}^k\phi_{a,i}^*\circ \Psi_{\Lambda}^a$$
to be an $R$-linear map such that
$$\rho^{*(i,k)}(\Lambda^{-r}\mathbf{e}_{\vec{{\bf v}}^{(k)},\Lambda}^* )={\rm Pr}_{\Lambda}\circ {\rm Sym}^k\phi_{a,i}^*(\Lambda^{-qr}\mathbf{e}_{\vec{{\bf v}}^{(k)},\Lambda^q}^*),$$
where ${\rm Pr}_{\Lambda}$ is a projection map such that
$${\rm Pr}_{\Lambda}\big(\sum_{r\in \mathbb{Z}}a_r\Lambda^{-r}\mathbf{e}_{\vec{{\bf v}}^{(k)},\Lambda}^*\big)=\sum_{r\in \mathbb{Z}_{\ge 0}}a_r\Lambda^{-r}\mathbf{e}_{\vec{{\bf v}}^{(k)},\Lambda}^* .$$

\begin{lem}\label{lemma5}
$\rho^{*(i,k)}$ is an $R$-linear map from  $S_{k,i,\Lambda}^*(\mathcal{O}_0^*)$ to $S_{k,i,\Lambda}^*(\mathcal{O}_0^*)$.
\end{lem}
\begin{proof}
Choose an element in $\zeta^* \in S_{k,i,\Lambda}^*(\mathcal{O}_0^*)$. Let
$$\zeta^*=\sum_{r\in \mathbb{Z}_{\ge 0},\vec{{\bf v}}^{(k)}\in \mathcal{M}_{k}^{(i)}}a^*(r,\vec{{\bf v}}^{(k)})\pi^{-(n+1)r}\Lambda^{-r}\mathbf{e}_{\vec{{\bf v}}^{(k)},\Lambda}^*.$$
Then we have that
\begin{align*}
\rho^{*(i,k)}(\zeta^*)&={\rm Pr}_{\Lambda}\circ {\rm Sym}^k\phi_{a,i}^*\big(\sum_{r\in \mathbb{Z}_{\ge 0},\vec{{\bf v}}^{(k)}\in \mathcal{M}_{k}^{(i)}}a^*(r,\vec{{\bf v}}^{(k)})\pi^{-(n+1)r}\Lambda^{-qr}\mathbf{e}_{\vec{{\bf v}}^{(k)},\Lambda^q}^* \big) \\
&={\rm Pr}_{\Lambda}\circ \sum_{r\in \mathbb{Z}_{\ge 0},\vec{{\bf v}}^{(k)}\in \mathcal{M}_{k}^{(i)}}a^*(r,\vec{{\bf v}}^{(k)})\pi^{-(n+1)r}\Lambda^{-qr}{\rm Sym}^k\phi_{a,i}^*(\mathbf{e}_{\vec{{\bf v}}^{(k)},\Lambda^q}^*).
\end{align*}
Since ${\rm Sym}^k\phi_{a,i}^*$ is a map from ${\rm Sym}^k (\wedge^i\mathcal{H}_{\Lambda^q}(\mathcal{O}_{0,q}))$ to ${\rm Sym}^k (\wedge^i\mathcal{H}_{\Lambda}(\mathcal{O}_{0,q}))$, we write
$${\rm Sym}^k\phi_{a,i}^*(\mathbf{e}_{\vec{{\bf v}}^{(k)},\Lambda^q}^* )=\sum_{s\in \mathbb{Z}_{\ge 0},\vec{{\bf u}}^{(k)}\in \mathcal{M}_{k}^{(i)}}b^*(s,\vec{{\bf u}}^{(k)})\pi^{\frac{(n+1)s}{q}}\Lambda^s \mathbf{e}_{\vec{{\bf u}}^{(k)},\Lambda}^* ,$$
where $b^*(s,\vec{{\bf u}}^{(k)})\rightarrow 0$ as $s\rightarrow \infty$.
Hence
\begin{align*}
\rho^{*(i,k)}(\zeta^*)&=
{\rm Pr}_{\Lambda}\Big(\sum_{s,r\in \mathbb{Z}_{\ge 0},\vec{{\bf v}}^{(k)},\vec{{\bf u}}^{(k)}\in \mathcal{M}_{k}^{(i)}}a^*(r,\vec{{\bf v}}^{(k)})b^*(s,\vec{{\bf u}}^{(k)})\pi^{-(n+1)r+\frac{(n+1)s}{q}}\Lambda^{s-qr}\mathbf{e}_{\vec{{\bf u}}^{(k)},\Lambda}^*\Big) \\
&=\sum_{\tau\in \mathbb{Z}_{\ge 0},\vec{{\bf u}}^{(k)}\in \mathcal{M}_{k}^{(i)}}c(\tau,\vec{{\bf u}}^{(k)})\pi^{\frac{-(n+1)\tau}{q}}\Lambda^{-\tau}\mathbf{e}_{\vec{{\bf u}}^{(k)},\Lambda}^*,
\end{align*}
where $$c(\tau,\vec{{\bf u}}^{(k)})=\sum_{qr-s=\tau\in \mathbb{Z}_{\ge 0},\vec{{\bf v}}^{(k)}\in \mathcal{M}_{k}^{(i)}}a^*(r,\vec{{\bf v}}^{(k)})b^*(s,\vec{{\bf u}}^{(k)}) .$$ Observe that $c(\tau,\vec{{\bf u}}^{(k)})\in R$ since $|\mathcal{M}_{k}^{(i)}|$ is finite, $a^*(r,\vec{{\bf v}}^{(k)}), b^*(s,\vec{{\bf u}}^{(k)})\in R$  and $b^*(s,\vec{{\bf u}}^{(k)})\rightarrow 0 $ as $s\rightarrow \infty$.
It follows that $\rho^{*(i,k)}(\zeta^*)\in S_{k,i,\Lambda}^*(\mathcal{O}_{0,q}^*)\subset S_{k,i,\Lambda}^*(\mathcal{O}_{0}^*)$.
 The proof of Lemma \ref{lemma5} is finished.
\end{proof}

Define a pairing $\langle,\rangle_k: {\rm Sym}^k(\wedge^i\mathcal{H}_{\Lambda}) \times S_{k,i,\Lambda}^*(\mathcal{O}_0^*) \rightarrow \Omega$ as follows.
For $$\zeta=\sum_{r\in \mathbb{Z},\vec{{\bf v}}^{(k)}\in \mathcal{M}_{k}^{(i)}} a(r,\vec{{\bf v}}^{(k)})\pi^{(n+1)r}\Lambda^{r}e_{\vec{v}_{j_1}^{(i)}}\cdots e_{\vec{v}_{j_k}^{(i)}} \in {\rm Sym}^k(\wedge^i\mathcal{H}_{\Lambda})$$ and
$$\zeta^*=\sum_{s\in \mathbb{Z}_{\ge 0},\vec{{\bf u}}^{(k)}\in \mathcal{M}_{k}^{(i)}} a^*(s,\vec{{\bf u}}^{(k)})\pi^{-(n+1)s}\Lambda^{-s}e_{\vec{u}_{j_1}^{(i)},\Lambda}^{*}\cdots e_{\vec{u}_{j_k}^{(i)},\Lambda}^{*}  \in S_{k,i,\Lambda}^*(\mathcal{O}_0^*),$$
let
 $$ \langle\zeta,\zeta^*\rangle_k:=\sum_{r\in \mathbb{Z}_{\ge 0},\vec{{\bf v}}^{(k)}\in \mathcal{M}_{k}^{(i)}} a(r,\vec{{\bf v}}^{(k)})a^*(r,\vec{{\bf v}}^{(k)})\langle e_{\vec{v}_{j_1}^{(i)}}\cdots e_{\vec{v}_{j_k}^{(i)}},e_{\vec{v}_{j_1}^{(i)},\Lambda}^{*}\cdots e_{\vec{v}_{j_k}^{(i)},\Lambda}^{*}\rangle_k.$$
Observe that $\langle e_{\vec{v}_{j_1}^{(i)}}\cdots e_{\vec{v}_{j_k}^{(i)}},e_{\vec{v}_{j_1}^{(i)},\Lambda}^{*}\cdots e_{\vec{v}_{j_k}^{(i)},\Lambda}^{*}\rangle_k$ is a rational number with $p$-adic valuation bounded below by $-k/(p-1)$. Since $a(r,\vec{{\bf v}}^{(k)})\rightarrow 0$ as $r\rightarrow \infty$ and $a^*(r,\vec{{\bf v}}^{(k)})$ is bounded,  $\langle, \rangle_k$ is well defined.

Let $\rho^{(i,k)}:=\psi_{\Lambda}^a\circ {\rm Sym}^k \phi_{a,i}$. Applying (\ref{eqn003}) and the same argument to \cite[Lemma 4.4]{HS17}, we obtain:
\begin{lem}\label{lemma4}
For positive integer $k$, $\zeta\in {\rm Sym}^k(\wedge^i\mathcal{H}_{\Lambda})$ and $\zeta^*\in S_{k,i,\Lambda}^*(\mathcal{O}_0^*)$,
$$ \langle\rho^{(i,k)}\zeta,\zeta^* \rangle_k=\langle\zeta,\rho^{*(i,k)}\zeta^* \rangle_k.$$
\end{lem}

For $1\le j\le n$, define the pairing of the tensor products ${\rm Sym}^{k_s'}(\wedge^{j}\mathcal{H}_{\Lambda})\otimes {\rm Sym}^{k_s}(\wedge^{j+1}\mathcal{H}_{\Lambda})$ and $ S_{k_s',j,\Lambda}^*(\mathcal{O}_0^*)\otimes S_{k_s,j+1,\Lambda}^*(\mathcal{O}_0^*)$ by
\begin{equation*}
\langle\zeta_1\otimes \zeta_2,\zeta_1^*\otimes \zeta_2^*\rangle:=\langle\zeta_1,\zeta_1^*\rangle_{k_s'}\cdot \langle\zeta_2, \zeta_2^*\rangle_{k_s}.
\end{equation*}
Let
$$\beta^{(j,k_s',k_s)}:=\rho^{(j,k_s')}\otimes \rho^{(j+1,k_s)}\ {\rm and}\ \beta^{*(j,k_s',k_s)}:=\rho^{*(j,k_s')}\otimes \rho^{*(j+1,k_s)}$$ be the endomorphisms on ${\rm Sym}^{k_s'}(\wedge^j\mathcal{H}_{\Lambda})\otimes {\rm Sym}^{k_s}(\wedge^{j+1}\mathcal{H}_{\Lambda})$ and $ S_{k_s',j,\Lambda}^*(\mathcal{O}_0^*)\otimes S_{k_s,j+1,\Lambda}^*(\mathcal{O}_0^*)$, respectively.

From Lemma \ref{lemma4}, we deduce that
\begin{align*}
\langle\zeta_1\otimes \zeta_2, \beta^{*(j,k_s',k_s)}(\zeta_1^*\otimes \zeta_2^*)\rangle=
\langle\beta^{(j,k_s',k_s)}(\zeta_1\otimes \zeta_2),\zeta_1^*\otimes \zeta_2^*\rangle.
\end{align*}
It then follows that
\begin{equation}\label{eqn07}
\det(1-\beta^{*(j,k_s',k_s)}T)= \det(1-\beta^{(j,k_s',k_s)}T)
\end{equation}
for $j\ge 1$.
When $j=0$, let $\beta^{*(0,k_s)}:= \rho^{*(1,k_s)}$ and $\beta^{(0,k_s)}:= \rho^{(1,k_s)}$. Then (\ref{eqn07}) still holds. Now we can give the proof of our main result.

\begin{proof}[ Proof of Theorem \ref{thm1}.]
We only prove the case $j\ge 1$. A similar proof can be applied to the case $j=0$.
By Corollary \ref{coro}, we have 
\begin{align*}
\phi_{a,j}^*(e_{\vec{u}_0^{(j)},\Lambda^q}^{*})=e_{\vec{u}_0^{(j)},\Lambda}^{*}+\Lambda \omega
\end{align*}
 for some $\omega\in \wedge^j\mathcal{H}_{\Lambda}^*(\mathcal{O}_{0,q})$.
For positive integer $k$, we have that
$$ {\rm Sym}^k \phi_{a,j}^*(e_{\vec{u}_0^{(j)},\Lambda^q}^{*})^k=\phi_{a,j}^*(e_{\vec{u}_0^{(j)},\Lambda^q}^{*} )^k=(e_{\vec{u}_0^{(j)},\Lambda}^{*}+\Lambda \omega)^k=
(e_{\vec{u}_0^{(j)},\Lambda}^{*})^k+\Lambda \omega'$$
for some $\omega'\in {\rm Sym}^k (\wedge^j\mathcal{H}_{\Lambda}(\mathcal{O}_{0,q}))$.
Hence
\begin{align*}
&\beta^{*(j,k_s',k_s)}\big((e_{\vec{u}_0^{(j)},\Lambda}^{*} )^{k_s'}\otimes (e_{\vec{u}_0^{(j+1)},\Lambda}^{*})^{k_s}\big)\\
=&\big({\rm Pr}_{\Lambda}\circ {\rm Sym}^{k_s'}\phi_{a,j}^*(e_{\vec{u}_0^{(j)},\Lambda^q}^{*})^{k_s'}\big)\otimes \big({\rm Pr}_{\Lambda}\circ {\rm Sym}^{k_s}\phi_{a,j+1}^*(e_{\vec{u}_0^{(j+1)},\Lambda^q}^{*})^{k_s}\big)
\\
=&(e_{\vec{u}_0^{(j)},\Lambda}^{*})^{k_s'}\otimes (e_{\vec{u}_0^{(j+1)},\Lambda}^{*})^{k_s}.
\end{align*}
It follows that $T=1$ is a characteristic root of $\det(1-\beta^{*(j,k_s',k_s)}T)$. Then by (\ref{eqn07}) that $T=1$ is a root of $\det(1-\beta^{(j,k_s',k_s)}T)$.
Note that $\lim_{s\rightarrow \infty} \beta^{(j,k_s',k_s)}=\beta^{(j,\infty,\kappa)}$.
 Then $T=1$ is also a root of   $\det(1-\beta^{(j,\infty,\kappa)}T)$.
It follows from Theorem \ref{thm2} that $T=1$ is the unique root of $L_{unit}(j,k,T)$. This finishes the proof of Theorem \ref{thm1}.

\end{proof}

%\begin{center}
%{\sc Acknowledgements}
%\end{center}

\bibliographystyle{amsplain}

\end{document}